\documentclass[11pt]{amsart}
\usepackage{amsbsy,amsfonts,amsmath,amsopn,amssymb,amsthm,amsxtra}
\usepackage{bm,bbm,color,enumerate,float,fullpage,geometry,graphicx,ifthen}
\usepackage{mathrsfs,multirow,pbox,pgf,tikz,tikz-cd,upref,url,xcolor}
\usepackage[utf8]{inputenc}
\usepackage[english]{babel}
\usepackage[all,color]{xy}
\usepackage{hyperref}
\hypersetup{
	colorlinks,
	final,
	pdfauthor={M Aguiar, J Bastidas},
	pdftitle={Associahedra, Cyclohedra and inversion of power series},
	pdfkeywords={Hopf monoids, antipode, graph associahedra, generalized permutahedra, associahedron, cyclohedron, Lagrange inversion},
	linkcolor = cyan,
	citecolor = magenta,
}
\usepackage[square,sort,comma,numbers]{natbib}
\usepackage{array} \newcolumntype{L}{>{\displaystyle}l} \newcolumntype{C}{>{\displaystyle}c} \newcolumntype{R}{>{\displaystyle}r}
\usetikzlibrary{shapes.multipart,positioning,decorations.pathreplacing,calc,intersections,folding}

\newtheorem{theorem}{Theorem}[section]
\newtheorem{proposition}[theorem]{Proposition}
\newtheorem{lemma}[theorem]{Lemma}
\newtheorem{corollary}[theorem]{Corollary}
\theoremstyle{definition}

\newtheorem{example}[theorem]{Example}
\newtheorem{problem}[theorem]{Problem}
\theoremstyle{remark}
\newtheorem{remark}{Remark}
\numberwithin{equation}{section}

\usepackage{enumitem}
\setlist{leftmargin=5.5mm}

\DeclareMathOperator{\Id}{Id}
\newcommand{\Kb}{\Bbbk}
\newcommand{\cB}{\mathcal{B}}
\newcommand{\RR}{\mathbb{R}}
\newcommand{\opp}[1]{\overline{#1}}
\newcommand{\GP}{\mathrm{GP}}
\newcommand{\hc}{\mathrm{C}}
\newcommand{\rmF}{\mathrm{F}}
\newcommand{\rmA}{\mathrm{A}}
\newcommand{\rmP}{\mathrm{P}}
\newcommand{\rmW}{\mathrm{W}}
\newcommand{\bbX}{\mathbb{X}}
\newcommand{\p}{\mathfrak{p}}
\newcommand{\q}{\mathfrak{q}}
\newcommand{\f}{\mathfrak{f}}
\newcommand{\ass}{\mathfrak{a}}
\newcommand{\cyc}{\mathfrak{c}}
\renewcommand{\path}{\mathtt{p}}
\newcommand{\apode}{\mathbbm{s}}
\newcommand{\cycle}{\mathtt{c}}
\newcommand{\qqand}{\qquad\text{and}\qquad}
\newcommand{\face}{\Sigma}

\definecolor{colblue}{HTML}{007CFC}
\definecolor{colgreen}{HTML}{7DC636}
\definecolor{colred}{HTML}{D20000}
\definecolor{colorange}{HTML}{F59723}
\newcommand\red[1]{{\color{colred}{#1}}}
\newcommand\blue[1]{{\color{colblue}{#1}}}


\begin{document}

\title{Associahedra, Cyclohedra and inversion of power series}

\author[M.~Aguiar]{Marcelo Aguiar}
\address{Department of Mathematics\\
Cornell University\\
Ithaca, NY 14853}
\email{\href{mailto:maguiar@math.cornell.edu}{maguiar@math.cornell.edu}}
\urladdr{\url{http://www.math.cornell.edu/~maguiar}}

\author[J.~Bastidas]{Jose Bastidas}
\address{Department of Mathematics\\
Cornell University\\
Ithaca, NY 14853}
\email{\href{mailto:bastidas@math.cornell.edu}{bastidas@math.cornell.edu}}
\urladdr{\url{https://sites.google.com/view/bastidas}}

\begin{abstract}
We introduce the Hopf monoid of sets of cycles and paths, which contains the Faà di Bruno Hopf monoid as a submonoid. We give cancellation-free and grouping-free formulas for its antipode, one in terms of tubings and one in terms of \emph{pointed} noncrossing partitions. We provide an explicit description of the group of characters of this Hopf monoid in terms of pairs of power series. Using graph associahedra, we relate paths and cycles to associahedra and cyclohedra, respectively. We give formulas for inversion in the group of characters in terms of the faces of these polytopes.
\end{abstract}

\maketitle

\section*{Introduction}

Combinatorial species~\citep{joyal} provide a unified framework to study families of combinatorial objects. When a family of combinatorial objects has compatible operations to \emph{merge} and \emph{break} structures, they can be studied under the scope of \emph{Hopf monoids} in the category of species~\citep{am10,am13}. Aguiar and Ardila~\citep{aa17} introduced the Hopf monoid of generalized permutahedra~$\GP$, which encompasses several Hopf monoids that had previously been studied on a case by case basis. The quotient of~$\GP$ modulo normal equivalence of polytopes is denoted~$\opp{\GP}$.

The submonoid~$\opp{\rmA}$ of~$\opp{\GP}$ generated by (classes of) associahedra is one of the main constructions in~\citep{aa17}. Its group of characters~$\bbX(\opp{\rmA})$ is isomorphic to the group of power series~$f(x) = x + a_1x^2 + a_2x^3 + \dots$ under composition. By solving the antipode problem for~$\GP$, and therefore for~$\opp{\rmA}$, they explained the connection between the classical Lagrange inversion formula and the face structure of the associahedron, thus solving a question posted by Loday.

The main goal of this paper is to develop a parallel result simultaneously involving the associahedron and the cyclohedron.

Low dimensional cyclohedra and associahedra are isomorphic. To bypass this technical difficulty, we work in the more general context of graphs. We introduce the Hopf monoid of sets of paths and cycles~$\hc$. Structures in~$\hc$ are (not necessarily simple) graphs whose connected components are paths or cycles. We give cancellation-free and grouping-free formulas for the antipode of~$\hc$. The former in terms of tubings (Theorem~\ref{t:apode_cycles_tub}), and the later in terms of \emph{pointed} noncrossing partitions (Theorem~\ref{t:apode_cycles_Cat}). We compute the character group~$\bbX(\hc)$ of this Hopf monoid.

\newtheorem*{t:charC}{Theorem~\ref{t:chars_hc}}
\begin{t:charC}
The group of characters of~$\hc$ is isomorphic to the group of pairs of power series
$$
\big(g(x),h(x)\big) = \Big( x + \sum_{n\geq 1} a_n x^{n+1},\sum_{n\geq 1} c_n \dfrac{x^n}{n} \Big)
$$
with product defined by
$$
\big(g_1(x),h_1(x)\big) \cdot \big(g_2(x),h_2(x)\big) = \big( g_1(g_2(x)) , h_1(g_2(x)) + h_2(x) \big).
$$
\end{t:charC}

Using graph associahedra~\citep{cd06grass}, we relate this result to the character group of the Hopf submonoid of~$\opp{\GP}$ generated by associahedra and cyclohedra.

This document is structured as follows. The necessary tools from the theory of Hopf monoids, generalized permutahedra, and graph associahedra are reviewed in Section~\ref{s:prelims}. We introduce the Hopf monoid of sets of cycles and paths~$\hc$ in Section~\ref{s:hc}

and solve the antipode problem for~$\hc$ in Section~\ref{s:apode_C}.

Section~\ref{s:chars_C} describes the group of characters of this Hopf monoid.

In Section~\ref{s:ass_cyc_ps}, we combine the result of the previous sections to obtain formulas for inversion in this group of pairs of power series in terms of the face structure of associahedra and cyclohedra. This allows us to present some interesting new combinatorial identities involving pointed noncrossing partitions and Catalan numbers.

The investigation of the group of characters of cyclohedra was undertaken independently by our colleagues and friends Federico Ardila, Carolina Benedetti, and Rafael González D'León. They have informed us of obtaining similar results to ours. We thank them for their graciousness and appreciation of our work.

\section{Preliminaries}\label{s:prelims}

\subsection{Generalized permutahedra}

Let~$V$ be a finite dimensional real vector space endowed with an inner product~$\langle \,\cdot\, , \cdot\, \rangle$.
For a polytope~$\p \subseteq V$ and a vector~$v \in V$, let~$\p_v$ denote the face of~$\p$ maximized in the direction~$v$.
That is,
$
\p_v = \{ p \in \p \,:\, \langle p , v \rangle \geq \langle q , v \rangle \text{ for all } q \in \p\}.
$

The (outer) \textbf{normal cone} of a face~$\f$ of~$\p$ is the polyhedral cone
$
N(\f,\p) = \{ v \in V \mid \f \leq \p_v\}.
$
The \textbf{normal fan} of~$\p$ is the collection~$\face_\p = \{ N(\f,\p) \,:\, \f \leq \p\}$ of normal cones of~$\p$.
There is a natural order-reversing correspondence between faces of~$\p$ and cones in~$\face_\p$, as illustrated in Figure~\ref{f:poly_max}.

\begin{figure}[ht]
\includegraphics[scale=1]{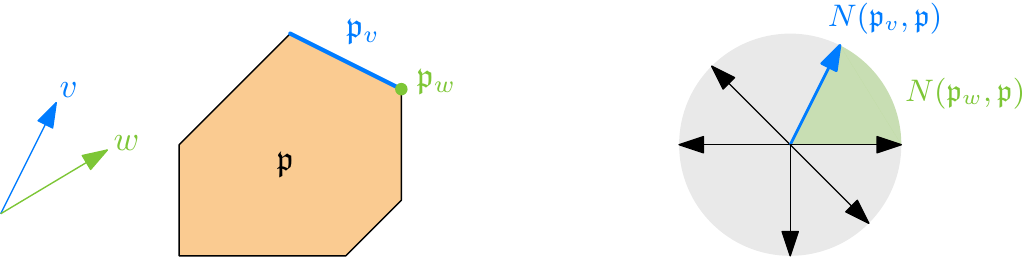}
\caption{A 2-dimensional polytope~$\p$ and two of its faces~$\p_v,\p_w$ maximized in directions~$v,w$, respectively.
			On the right, the normal fan~$\face_\p$.}\label{f:poly_max}
\end{figure}

Two polytopes~$\p$ and~$\q$ are said to be \textbf{normally equivalent}, denoted~$\p \equiv \q$, if~$\face_\p = \face_\q$.
If, on the other hand,~$\face_\p$ \textbf{refines}~$\face_\q$, we say that~$\q$ is a \textbf{deformation} of~$\p$.
Recall that a fan~$\face$ refines~$\face'$ if every cone in~$\face'$ is a union of cones in~$\face$.

Fix a finite set~$I$ and let~$\{e_i \mid i \in I\}$ denote the standard basis of~$\RR^I$.
A \textbf{composition} of~$I$ is a tuple~$F = (S_1 , \dots , S_k)$ of nonempty sets such that~$I = S_1 \sqcup \dots \sqcup S_k$, the sets~$S_i$ are the \textbf{blocks} of~$F$.
We say a block~$S_i$ (weakly) \textbf{precedes}~$S_j$ if it appears before in~$F$ (or if they are equal).
We associate the following polyhedral cone to a composition~$F$
$$
\sigma_F =
\{ x \in \RR^I \,:\,
x_i \geq x_j \text{ if the block containing } i \text{ weakly precedes the block containing } j \}.
$$
The \textbf{braid fan}~$\cB_I$ is the polyhedral fan consisting of all the cones of the form~$\sigma_F$.

\begin{figure}[h]
\includegraphics[scale=.8]{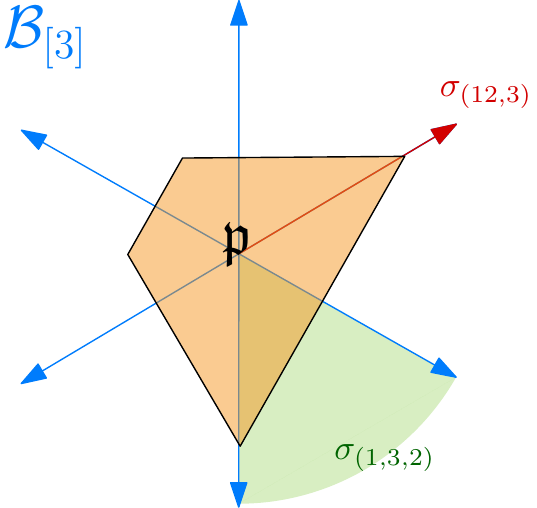}
\caption{The braid fan~$\cB_{[3]}$ and a generalized permutahedron~$\p$ in~$\RR^{[3]}$. Two cones of~$\cB_{[3]}$ are labeled by the corresponding composition of~$[3]$. For instance, the cone~$\sigma_{(12,3)}$ in red consists of those points~$x$ satisfying~$x_1 = x_2 \geq x_3$.}
\end{figure}

A \textbf{generalized permutahedron} of~$I$ is a polytope~$\p \subseteq \RR^I$ whose normal fan is refined by the braid fan~$\cB_I$.
This implies that the face of a generalized permutahedron~$\p \subseteq \RR^I$ maximized in some direction~$\sum_{i \in I} a_i e_i$ depends only on the relative order of the values~$a_i$, and not on the values themselves.

For example, the face of a generalized permutahedron~$\p \subseteq \RR^3$ maximized in the direction~$(1,2,2)$ is the same as the face maximized in the direction~$(-3,1,1)$, but potentially different from the one maximized by~$(2,1,2)$.

Generalized permutahedra were first introduced by Edmonds~\citep{edmonds} in the context of optimization of submodular functions, and have been of central interest to combinatorialists because they serve as polyhedral models for several families of combinatorial structures. They have been extensively studied by Postnikov~\citep{postnikov09}, Postnikov, Reiner, and Williams~\citep{prw08faces}, and many others. Aguiar and Ardila~\citep{aa17} endowed the family of generalized permutahedra with the structure of a Hopf monoid in the category of species. We review their construction in the following section.

\subsection{Hopf monoids and generalized permutahedra}

Combinatorial species were originally introduced by Joyal in~\citep{joyal} as a tool for studying generating power series from a combinatorial perspective.

Aguiar and Mahajan~\citep{am10,am13} studied monoidal structures on the category of species, in particular Hopf monoids, and exploited this rich algebraic structure to obtain outstanding combinatorial results.
This section presents a concise introduction to Hopf monoids in species.

\subsubsection{Hopf monoids}

Let~${\sf set}^\times$ denote the category of finite sets with bijections as morphisms, and~${\sf Set}$ the category of sets and arbitrary set functions.
A \textbf{(set) species}~$\rmP$ is a functor~${\sf set}^\times \rightarrow {\sf Set}$.
Explicitly, a species~$\rmP$ consists of the following data:
\begin{itemize}
\item For each finite set~$I$, a set~$\rmP[I]$.
		Elements~$x \in \rmP[I]$ are called \textbf{structures} of type~$\rmP$ on~$I$.
\item For each bijection~$\sigma : I \rightarrow J$, a map~$\rmP[\sigma] : \rmP[I] \rightarrow \rmP[J]$.
\end{itemize}
These maps satisfy~$\rmP[\sigma \circ \tau] = \rmP[\sigma] \circ \rmP[\tau]$ and~$\rmP[\Id] = \Id$.
Two structures~$x \in \rmP[I]$ and~$y \in \rmP[J]$ are \textbf{isomorphic} if~$y = \rmP[\sigma](x)$ for some bijection~$\sigma : I \rightarrow J$. We denote this by~$x \cong y$.

A species~$\rmP$ is said to be \textbf{connected} if~$\rmP[\emptyset]$ consists of exactly one element.
In this document we will only work with connected species, and let~$\epsilon$ denote the only structure in~$\rmP[\emptyset]$.
As it is customary, we write~$[n] = \{1,2,\dots,n\}$ and~$\rmP[n] := \rmP[[n]]$.

A species~$\rmP$ is a \textbf{Hopf monoid} if for any finite set~$I$ and decomposition~$I = S \sqcup T$, there are maps
$$
\begin{gathered}
\begin{array}{CCCC}
\mu_{S,T} : & \rmP[S] \times \rmP[T] & \rightarrow & \rmP[I] \\
& (x,y) & \mapsto & x \cdot y
\end{array}
\end{gathered}
\qqand
\begin{gathered}
\begin{array}{CCCC}
\Delta_{S,T} : & \rmP[I] & \rightarrow & \rmP[S] \times \rmP[T] \\
& z & \mapsto & (z|_S , z/_S)
\end{array}
\end{gathered}
$$
satisfying certain (co)unitality, (co)associativity, naturality and compatibility axioms.

Suppose~$I = S \sqcup T = J \sqcup K$ are two decompositions of the same set~$I$, and we have structures~$x \in \rmP[J]$ and~$y \in \rmP[K]$. Then, the compatibility axiom requires that
\begin{equation}\label{eq:comp}
(x \cdot y)|_S = x|_A \cdot y|_B
\qqand
(x \cdot y)/_S = x/_A \cdot y/_B,
\end{equation}
where~$A = S \cap J$ and~$B = S \cap K$.

The following is a graphical representation of~\eqref{eq:comp}:
$$
\includegraphics[scale=1]{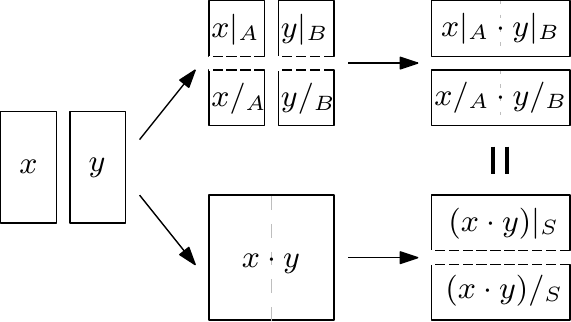}
$$

The \textbf{antipode}~$\apode$ of a Hopf monoid~$\rmP$ is the collection of linear maps
$
\apode_I : \Kb \rmP[I] \rightarrow \Kb \rmP[I],
$
defined inductively by:
		\begin{equation}\label{eq:mil_moo}
		\apode_I(x) = - \sum_{\substack{I = S \sqcup T \\ S \neq \emptyset}} x|_S \cdot \apode_T(x/_S)
		\end{equation}
for all~$x \in \rmP[I]$ with~$I \neq \emptyset$, and~$\apode_\emptyset = \Id$.

We refer to~\eqref{eq:mil_moo} as the Milnor and Moore formula.
In order to give meaning to the terms in the sum above, we have linearly extended the product~$\mu_{S,T}$ to get maps~$\Kb \rmP[S] \times \Kb \rmP[T] \rightarrow \Kb \rmP[I]$.

\begin{remark}
The antipode might not exists for a species that is not connected, even if it has compatible product and coproduct maps. In the case of connected species, the Milnor and Moore formula can be deduced from the more general definition of antipode \citep[Section 8.4]{am10}.
\end{remark}

In general, many cancellations occur among the~$2^{|I|}-1$ terms in~\eqref{eq:mil_moo}.
The antipode problem asks us to find a cancellation-free expression for the antipode of a given Hopf monoid.

\begin{problem}[The antipode problem]
Given a Hopf monoid~$\rmP$, find an explicit cancellation-free formula for the antipode~$\apode$ of~$\rmP$.
\end{problem}

The antipode is the Hopf monoid analog of the inverse map of a group. Just like the inverse map, it \emph{reverses} products~\citep[Proposition 1.22]{am10}.
That is, if~$I = S \sqcup T$,~$x \in \rmP[S]$ and~$y \in \rmP[T]$, then
\begin{equation}\label{eq:apode_prod}
\apode_I(x \cdot y) = \apode_T(y) \cdot \apode_S(x).
\end{equation}

\subsubsection{Characters}
A \textbf{character}~$\zeta$ of a Hopf monoid~$\rmP$ is a collection of linear maps~$\zeta_I : \Kb \rmP[I] \rightarrow \Kb$ satisfying
$$
\zeta_\emptyset(\epsilon) = 1\qquad\qquad \zeta_I(x \cdot y) = \zeta_S(x) \zeta_T(y)\qquad\qquad \zeta_I(x) = \zeta_J(y), \text{ whenever } x \cong y.
$$

The collection of characters~$\bbX(\rmP)$ of a connected Hopf monoid~$\rmP$ forms a group under \textbf{convolution product}.
If~$\zeta,\xi \in \bbX(\rmP)$ are two characters, their convolution~$\zeta*\xi$ is defined by
$$
(\zeta*\xi)_I(x) = \sum_{I = S \sqcup T} \zeta_S(x|_S)\xi_T(x/_S),
$$
for all~$x \in \rmP[I]$. The multiplicativity of~$\zeta*\xi$ follows from the compatibility axiom of~$\rmP$.
One easily verifies that the character~$u \in \bbX(\rmP)$ determined by
$$
u_I(x) = \begin{cases}
1 & \text{if } x = \epsilon,\\
0 & \text{if } I \neq \emptyset,
\end{cases}
$$
is the unit with respect to the convolution product.

\begin{theorem}[{\citep{am10}}]\label{t:inverse_chars}
The inverse of a character~$\zeta$ is given by the collection of maps
$$
\zeta^{*-1}_I = \zeta_I \circ \apode_I.
$$
\end{theorem}

\subsubsection{The Hopf monoid of generalized permutahedra}

Let~$\GP$ denote the \textbf{Hopf monoid of generalized permutahedra}. As a species,~$\GP[I]$ consists of all generalized permutahedra in~$\RR^I$. It forms a Hopf monoid with the the following operations:
\begin{itemize}
\item If~$I = S \sqcup T$,~$\p \in \GP[S]$ and~$\q \in \GP[T]$, the Cartesian product~$\p \times \q \subseteq \RR^S \times \RR^T = \RR^I$ is a generalized permutahedron.
		The product of~$\GP$ is defined by
		$$
		\p \cdot \q
		= \p \times \q.
		$$
\item If~$I = S \sqcup T$ and~$\p \in \GP[I]$, the face of~$\p$ maximized in the direction~$\sum_{i \in S} e_i$ decomposes as a product~$\p|_S \times \p/_S$, where~$\p|_S \in \RR^S$ and~$\p/_S \in \RR^T$ are generalized permutahedra.
		The coproduct is defined by
		$$
		\Delta_{S,T}(\p) = (\p|_S,\p/_S).
		$$
\end{itemize}

The coassociativity of~$\GP$ is equivalent to the following \emph{greediness} property for generalized permutahedra.

\begin{proposition}\label{p:greedy}
Let~$\p \subseteq \RR^I$ be a generalized permutahedron and~$v = \sum_{i \in I} a_i e_i \in \RR^I$ be an arbitrary direction.
Let~$S_1, S_2, \dots, S_k$ be the composition of~$I$ such that the block containing~$i$ weakly precedes the block containing~$j$ whenever~$a_i \geq a_j$. Then, the face~$\p_v$ of~$\p$ maximized in the direction~$v$ can be obtained by sequentially maximizing~$\sum_{i \in S_l} e_i$, for~$l = 1,2\dots,k$.
\end{proposition}

Notably, Aguiar and Ardila found the following cancellation-free formula for the antipode of~$\GP$.

\begin{theorem}[{\citep[Theorem 7.1]{aa17}}]\label{t:apode_GP}
The following is a cancellation-free and grouping-free formula for the antipode of~$\GP$.
If~$\p \in \GP[I]$, then
$$
\apode_I(\p) = (-1)^{|I|} \sum_{\q \leq \p} (-1)^{\dim \q} \q,
$$
where the sum is over all the faces~$\q$ of~$\p$.
\end{theorem}

Let~$\opp{\GP}$ denote the species of generalized permutahedra modulo normal equivalence. It is a Hopf monoid quotient of~$\GP$. Therefore, the preceding antipode formula still holds in~$\opp{\GP}$, but it is no longer grouping-free. Different faces of~$\p$ might be normally equivalent. For instance, the standard permutahedron in~$\RR^3$ has~$13$ faces, but they correspond to only~$5$ normal equivalence classes.

\subsection{Graph associahedra}\label{ss:g_ass}

Given a graph~$g$ with vertex set~$I$, Carr and Devadoss~\citep{cd06grass} construct a polytope~$\ass_g \subseteq \RR^I$, called the \emph{graph associahedron} of~$g$, whose face poset relates to the connected subgraphs (tubes) of~$g$. Not long after, Postnikov~\citep{postnikov09} introduced a more general family of polytopes associated to \emph{nested sets}. These polytopes are constructed as certain Minkowski sums of simplices. We will follow the second approach for definitions.

Let~$g$ be a graph on vertex set~$I$. Given a subset~$S \subseteq I$, let~$g \!:\! S$ denote the \textbf{induced subgraph} on~$S$. That is,~$g \!:\! S$ consists of vertices~$S$ and those edges of~$g$ having both endpoints in~$S$. The \textbf{graph associahedron} of~$g$ is
$$
\ass_g = \sum_{\substack{S \subseteq I \\ g : S \text{ is connected}}} \Delta_S,
$$
where~$\Delta_S$ is the simplex with vertices~$e_i$ for~$i \in S$. The dimension of~$\ass_g$ is~$|I| - c(g)$, where~$c(g)$ denotes the number of connected components of~$g$. Observe that the definition of~$\ass_g$ does not distinguish multiple edges or loops.

Let~$I = S \sqcup T$ and~$u,v \in S$. A \textbf{$T$-thread} joining~$u$ and~$v$ is a path in~$g$ with endpoints~$u,v$ all whose intermediate vertices, if any, are in~$T$. Let~$g|_S$ be the graph on vertex set~$S$ with edges~$\{u,v\} \subseteq S$ whenever there is a~$T$-thread joining~$u$ and~$v$. In particular,~$g|_S$ contains all the edges of the induced subgraph~$g \!:\! S$.

\begin{proposition}[\citep{aa17}]\label{p:faces_g_ass}
Let~$I = S \sqcup T$.
If~$g_1$ is a graph with vertex set~$S$ and~$g_2$ is a graph with vertex set~$T$,
then
$$
\ass_{g_1 \sqcup g_2} = \ass_{g_1} \times \ass_{g_2}.
$$
On the other hand, if~$g$ is a graph with vertex set~$I$, then
$$
\ass_g|_S \equiv \ass_{g|_S}
\qqand
\ass_g/_S = \ass_{g : T}.
$$
\end{proposition}

Faces of the graphic associahedron~$\ass_g$ are in bijection with \emph{tubings} of~$g$.

A \textbf{tube} of~$g$ is a collection of vertices~$S \subseteq I$ such that the induced graph~$g \!:\! S$ is connected.
A~\textbf{tubing}~$t$ is a collection of tubes such that:
\begin{itemize}
\item if~$S_1,S_2 \in t$, then either~$S_1 \subseteq S_2$,~$S_2 \subseteq S_1$, or~$S_1 \cap S_2 = \emptyset$,
\item if~$S_1,\dots,S_k \in t$ are pairwise disjoint, then~$S_1 \sqcup \dots \sqcup S_k$ is not a tube of~$g$, and
\item every connected component of~$g$ is in~$t$.
\end{itemize}
Thus, every tubing of~$g$ contains at least~$c(g)$ tubes.

A tubing~$t$ induces a partition~$\pi(t) = \{ \pi_1 , \dots , \pi_k \}$ of~$I$ and a graph~$g(t)$ whose connected components are the blocks of~$\pi(t)$ as follows. Two vertices~$i,j \in I$ belong to the same block of~$\pi(t)$ if and only if they are contained in the same set of tubes of~$t$. For each~$\pi_i \in \pi(t)$, let~$S_i \in t$ be the minimum tube containing~$\pi_i$, which is well-defined since tubes in~$t$ are either nested or disjoint. Let~$g_i$ be the graph on node set~$\pi_i$ having an edge between~$u,v \in \pi_i$ if
\begin{itemize}
\item~$\{u,v\}$ is an edge in~$g$, or
\item there is a~$T$-thread joining~$u$ and~$v$ for a tube~$T \in t$ strictly contained in~$S_i$.
\end{itemize}
Define~$g(t) = g_1 \sqcup g_2 \sqcup \dots \sqcup g_k$. Note that each graph~$g_i$ is connected, so~$\pi(t)$ is precisely the partition of~$I$ into connected components of~$g(t)$.

\begin{example}
Let us now consider the graph~$g$ on vertex set~$[6]$ and the tubing~$t = \{ 34, 2345 , [6] \}$ shown below.
We omit the tube~$[6] \in t$ from the picture.
$$
g = \begin{gathered} \includegraphics[scale=1,page=1]{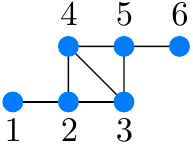} \end{gathered}
\qquad\qquad\qquad\qquad
t = \begin{gathered} \includegraphics[scale=1,page=2]{ex_tubes.pdf} \end{gathered}
$$
Then,~$\pi(t) = \{16,25,34\}$, and

$$
g(t) = \begin{gathered} \includegraphics[scale=1,page=3]{ex_tubes.pdf} \end{gathered}.
$$
Indeed,~$12356$ is a thread through the green tube joining~$1$ and~$6$, and~$245$ is a thread through the red tube joining~$2$ and~$5$.
\end{example}

The collection of all tubings of~$g$ is ordered by containment. That is,~$t \preceq t'$ if~$t$ can be obtained from~$t'$ by removing tubes. This collection has a minimum element: the tubing consisting of only the connected components of~$g$.

\begin{theorem}[{\citep{cd06grass,postnikov09}}]
Let~$g$ be a graph on vertex set~$I$.
There is a an order-reversing bijection between faces of~$\ass_g$ and tubings of~$g$.
If~$t$ is a tubing and~$F_t$ is the corresponding face, then~$\dim F_t = |I| - |t|$.
\end{theorem}

Explicitly,~$F_t$ is the face of~$\ass_g$ maximized by the linear functional~$-\sum_{i \in I} n_i x_i$, where~$n_i$ is the number of tubes in~$t$ containing the vertex~$i$. Moreover,~$F_t \equiv \ass_{g(t)}$.

\begin{example}
Consider the graph~$g = \begin{gathered}\includegraphics{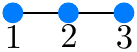}\end{gathered}$. The graph associahedron~$\ass_g$ is shown below.
\begin{center}
\includegraphics[width=.4\linewidth]{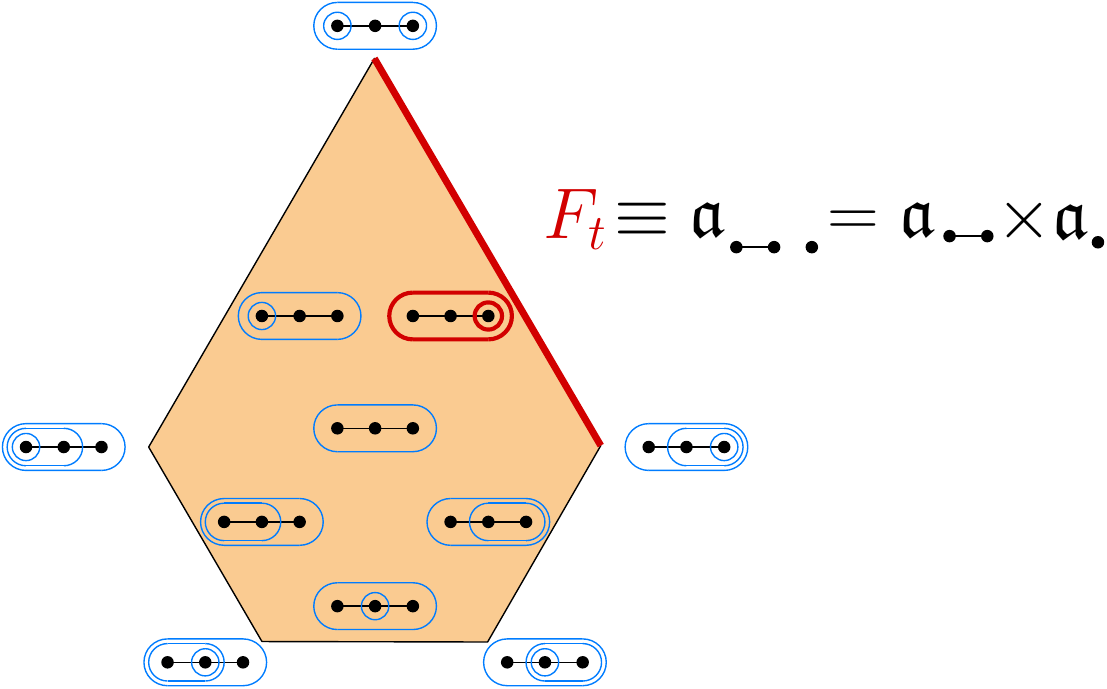}
\end{center}
The red face of~$\ass_g$ corresponds to the tubing~$t = \{3,123\}$. Therefore, it is the face where the functional~$-x_1 - x_2 - 2x_3$ is maximized or, equivalently, where~$x_1 + x_2$ is maximized. The value of~$x_1 + x_2$ along this face is~$5$, the total number of tubes of~$g$ intersecting the set~$\{1,2\}$.
\end{example}

\section{The Hopf monoid of sets of cycles and paths}\label{s:hc}

In this section we introduce the Hopf monoid of sets of cycles and paths~$\hc$. It naturally contains the \emph{Faà di Bruno} Hopf monoid (of sets paths)~$\rmF$ defined by Aguiar and Ardila in~\citep{aa17} as a submonoid and as a quotient.

The elements of~$\hc[I]$ are graphs with vertex set~$I$ whose connected components are either paths or cycles. We insist that a cycle has the same number of edges as vertices. Hence, cycles on~$1$ and~$2$ vertices are not simple graphs.

\begin{figure}[ht]
$$
\cycle_1 = \begin{gathered} \includegraphics[page=1]{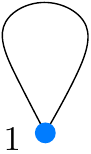} \end{gathered} \in \hc[1]
\qquad\qquad
\cycle_2 = \begin{gathered} \includegraphics[page=2]{Loops1_2.pdf} \end{gathered} \in \hc[2]
$$
\caption{These cycles are the only (isomorphism classes of) connected graphs in the species~$\hc$ that are not simple.}
\end{figure}

Following~\citep[Section 3.5]{aa17}, we write a path~$\path$ as a word, listing the vertices in the order they appear in~$\path$. Similarly, we write a cycle~$\cycle$ as a parenthesized word, listing the vertices in the cyclic order they appear in~$\cycle$. A set of paths and cycles~$g \in \hc[I]$ is denoted by~$w_1|\dots|w_r|(w'_1)|\dots|(w'_k)$, where each (parenthesized) word corresponds to a connected component of~$g$.

\begin{figure}[ht]
$$\includegraphics[page=1]{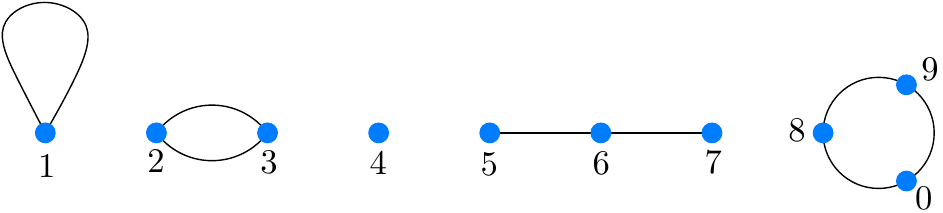}$$
\caption{This set of paths and cycles corresponds to~$4|567|(1)|(23)|(890) = 765|4|(32)|(1)|(908) = \dots$.}
\label{f:ex_words}
\end{figure}

We proceed to give~$\hc$ a Hopf monoid structure. The product is given by the disjoint union of graphs. For~$I = S \sqcup T$,~$g_1 \in \hc[S]$ and~$g_2 \in \hc[T]$, define
$$
\mu_{S,T}(g_1,g_2) = g_1 \cdot g_2 = g_1 \sqcup g_2.
$$
The definition of coproduct is a bit more subtle. For~$g \in \hc[I]$ and~$I = S \sqcup T$, define
$$
\Delta_{S,T}(g) = (g|_S,g/_S) = (g|_S,g \!:\! T)
$$
where

$g|_S$ is defined just like before Proposition~\ref{p:faces_g_ass} with a small caveat: a~$T$-thread might have the same start and end point and different~$T$-threads yield different edges. This guarantees that if~$S$ intersects a cycle of~$g$ in~$1$ or~$2$ vertices, the corresponding in~$g|_S$ is again a cycle.

\begin{example}
If~$g = 4|567|(1)|(23)|(890)$ is the example in Figure~\ref{f:ex_words} and~$S = \{1,4,5,7,9\}$, then
$$
g|_S = \begin{gathered} \includegraphics[page=4]{ex_words.pdf} \end{gathered}
\qqand
g/_S = \begin{gathered} \includegraphics[page=5]{ex_words.pdf} \end{gathered}
$$
Equivalently,~$g|_S = 4|57|(1)|(9)$ and~$g/_S = 6|08|(23)$. Dashed edges in~$g|_S$ correspond to~$T$-threads where~$T = I \setminus S = \{2,3,6,8,0\}$.
\end{example}

Let~$\rmF$ be the subspecies of~$\hc$ consisting of graphs whose connected components are paths.
It is naturally a Hopf submonoid of~$\hc$, and this structure coincides with the definition of Aguiar and Ardila in~\citep{aa17}.
Moreover,~$\rmF$ is also isomorphic to a Hopf monoid quotient of~$\hc$.
This follows from the following observation: if a graph~$g \in \hc[I]$ contains a cycle, then for any decomposition~$I = S \sqcup T$ one of~$g|_S$ and~$g/_S$ also contains a cycle.

The first fact implies that the antipode formula of~$\rmF$ is still valid for paths in~$\hc$.
The second fact implies that~$\bbX(\rmF)$ is a subgroup of~$\bbX(\hc)$.
Hence, the inversion formulas for~$\bbX(\rmF)$ deduced in~\citep{aa17} are still valid for~$\bbX(\hc)$.

\section{The antipode problem for~$\hc$}\label{s:apode_C}

In this section, we solve the antipode problem for the Hopf monoid~$\hc$. Aguiar and Ardila found a cancellation-free formula for the antipode of the submonoid~$\rmF$ in terms of tubings of paths. By grouping terms, they also describe a cancellation-free and grouping-free formula in terms of noncrossing partitions. Since the antipode axiom is expressed entirely in terms of products and coproducts, these formulas still hold for sets of paths in~$\hc$. We review their results first.

A \textbf{noncrossing partition} of a path~$\path \in \rmF[I] \subseteq \hc[I]$ is a partition~$\pi = \{\pi_1,\dots,\pi_k\}$ of~$I$ such that if~\red{$a,c \in \pi_i$} and~\blue{$b,d \in \pi_j$} distinct blocks~$\pi_i$ and~$\pi_j$, then~$\red{a}\blue{b}\red{c}\blue{d}$ do not appear in that order in~$\path$. If~$\pi$ is a noncrossing partition of~$\path$, let~$\path(\pi)$ be the following set of paths
$$
\path(\pi) = \path|_{\pi_1} \sqcup \path|_{\pi_2} \sqcup \dots \sqcup \path|_{\pi_k}.
$$
That is,~$\path(\pi)$ is the disjoint union of~$k$ paths. There is one path on each block~$\pi_i \in \pi$, and the vertices of~$\path|_{\pi_i}$ appear in the same relative order as in~$\path$.

\begin{example}\label{ex:NC1}
Consider the following path~$\path$ on~$12$ vertices and noncrossing partition~$\pi$ of~$\path$. Two vertices of~$\path$ are on the same block of~$\pi$ if they are connected by gray edges.
$$
\includegraphics[scale=1,page=1]{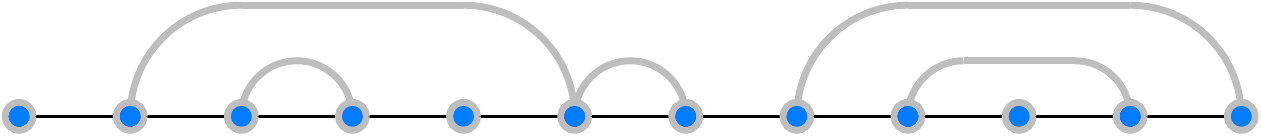}
$$
Then,~$\path(\pi)$ consists of: one path on~$3$ vertices, three paths on~$2$ vertices, and three paths on~$1$ vertex.

\end{example}

Let~$t$ be a tubing of a path~$\path$. The tubes of~$t$ are either nested or disjoint; hence, the associated partition~$\pi(t)$ defined in Section~\ref{ss:g_ass} is noncrossing. The corresponding graph~$\path(t)$ is precisely~$\path(\pi(t))$.

\begin{proposition}[{\citep[Proposition 25.2]{aa17}}]\label{p:apode_paths_tub}
The antipode of the Hopf monoid of paths~$\rmF$ is given by the following cancellation-free expression.
If~$\path$ is a path on~$I$,
$$
\apode_I(\path) = \sum_t (-1)^{|t|} \path(t).
$$
The sum is over all tubings of~$\path$.
\end{proposition}

The graph~$\path(t)$ is entirely determined by the noncrossing partition~$\pi(t)$. Hence, a grouping-free formula for the antipode of~$\rmF$ can be obtained by grouping tubings with the same associated noncrossing partition. To do so, we need to introduce the concept of adjacent closure of a noncrossing partition.

In order to simplify notation, let us consider the case of the standard path~$\path_n = 123\dots n \in \hc[n]$. Let~$\pi$ be a noncrossing partition of~$\path_n$. The \textbf{adjacent closure} of~$\pi$ is the noncrossing partition~$\opp{\pi}$ obtained by successively merging any two blocks~$\pi_i,\pi_j \in \pi$ such that~$\max \pi_i +1 = \min \pi_j$. For instance, if~$\pi = \{ 14, 23, 58 , 6 , 7 ,9 \}$, then~$\opp{\pi} = \{ 14589 , 23 , 67 \}$. This definition is easily generalized to any path, as the illustrated in the next example.

\begin{example}\label{ex:NC2}
Consider the path~$\path$ and noncrossing partition~$\pi$ in Example \ref{ex:NC1}.
The adjacent closure~$\opp{\pi}$ of~$\pi$ is shown in the following picture.
$$
\includegraphics[scale=1,page=2]{NC-ex.pdf}
$$
\end{example}

For a noncrossing partition~$\pi$, define
$$
C_{(\opp{\pi}:\pi)} = \prod_{\pi'_i \in \opp{\pi}} C_{n_i},
$$
where~$C_m$ is the~$m$-th Catalan number, and~$n_i$ is the number of blocks of~$\pi$ contained in~$\pi'_i$.

For the partition in Examples \ref{ex:NC1} and \ref{ex:NC2}, we have
$$
C_{(\opp{\pi}:\pi)} = C_3 C_2 C_1 C_1 = 5 \cdot 2 \cdot 1 \cdot 1 = 10.
$$

\begin{theorem}[{\citep[Theorem 25.4]{aa17}}]\label{t:apode_paths_cat}
The antipode of the Hopf monoid of paths~$\rmF$ is given by the following cancellation-free and grouping-free expression.
If~$\path$ is a path on~$I$,
$$
\apode_I(\path) = \sum_{\pi \in NC(\path)} (-1)^{|\pi|} C_{(\opp{\pi}:\pi)} \path(\pi),
$$
where~$NC(\path)$ is the set of all noncrossing partitions of~$\path$.
\end{theorem}

We now proceed to deduce results parallel to Proposition~\ref{p:apode_paths_tub} and Theorem~\ref{t:apode_paths_cat} for cycles.
A \textbf{pointed noncrossing partition} of a cycle~$\cycle = (i_1 i_2 \dots i_n) \in \hc[I]$ is a pair~$\pi = (\pi_0,\pi_+)$, where~$\{\pi_0\} \sqcup \pi_+ = \{ \pi_0, \pi_1,\dots,\pi_k \}$ is a noncrossing partition of the path~$i_1 i_2 \dots i_n$. It is a classical result that this is independent of the parenthesized word representing~$\cycle$. We refer to~$\pi_0$ as the \textbf{zero block} of~$\pi$, and to the blocks in~$\pi_+$ as the \textbf{nonzero blocks} of~$\pi$. Given a pointed noncrossing partition~$\pi = (\pi_0,\{\pi_1,\dots,\pi_k\})$ of~$\cycle$, we let~$\cycle(\pi)$ denote the following set of cycles and paths
$$
\cycle(\pi) = \cycle|_{\pi_0} \sqcup (\cycle/_{\pi_0})|_{\pi_1} \sqcup \dots \sqcup (\cycle/_{\pi_0})|_{\pi_k}.
$$

\begin{example}\label{ex:PNC1}
Consider the following cycle~$\cycle$ on~$12$ vertices and the pointed noncrossing partition~$(\pi_0,\pi_+)$.
$$
\includegraphics[scale=.8,page=1]{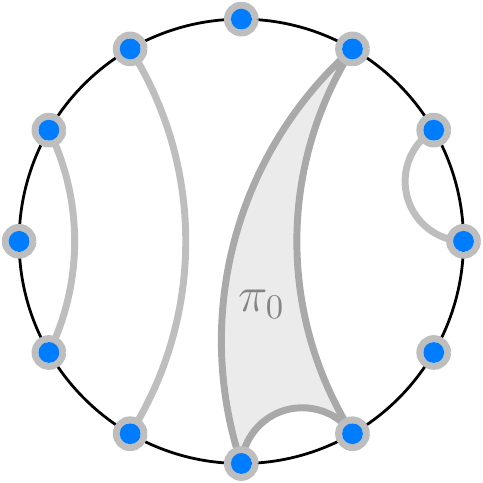}
$$
The zero block is shown in a darker shade of gray.
The graph~$\cycle(\pi)$ consists of: one cycle on~$3$ vertices ($\cycle|_{\pi_0}$), three paths on~$2$ vertices, and three paths on~$1$ vertex.
\end{example}

Let~$t$ be a tubing of a cycle~$\cycle$. It determines a pointed noncrossing partition~$\pi(t) = (\pi(t)_0 , \pi(t)_+)$ as follows. The underlying partition~$\{\pi(t)_0\} \sqcup \pi(t)_+$ is the partition associated to the tubing~$t$ as in Section~\ref{ss:g_ass}, and the zero-block~$\pi(t)_0$ corresponds to the elements contained only in the maximal tube~$I \in t$. Given a tubing~$t$ of~$\cycle$, define~$\cycle(t) = \cycle(\pi(t))$.

\begin{remark}
If~$\pi(t)_0$ contains at least 3 elements, the definition of~$\cycle(t)$ is exactly the same as the one in Section~\ref{ss:g_ass}.
However, when~$\pi(t)_0$ consists of 1 or 2 elements, we are emphasizing that the component in~$\pi(t)_0$ of~$\cycle(t)$ is a cycle (non-simple graph).
In terms of the corresponding graph associahedra, this does not make a difference.
\end{remark}

\begin{theorem}\label{t:apode_cycles_tub}
The following is a cancellation-free for the antipode of a cycle in~$\hc$.

If~$\cycle$ is a cycle on vertex set~$I$,
$$
\apode_I(\cycle) = \sum_t (-1)^{|t|} \cycle(t).
$$
The sum is over all tubings of~$\cycle$.
\end{theorem}

\begin{remark}
The reader might be familiar with the antipode formula of the ripping and sewing Hopf monoid~$\rmW$ in~\citep[Theorem 23.6]{aa17} which is also expressed in terms of tubings of a graph. However,~$\hc$ is \textbf{not} a submonoid of (the coopposite of)~$\rmW$, and Theorem~\ref{t:apode_cycles_tub} does not immediately follow from the antipode formula for~$\rmW$.
\end{remark}

\begin{proof}
Using Milnor and Moore formula~\eqref{eq:mil_moo} and~\eqref{eq:apode_prod}, we obtain
\begin{equation}\label{eq:apode_fst_step}
\apode_I(\cycle)
= - \sum_{\emptyset \subsetneq S \subseteq I} \cycle|_S \cdot \apode_T(\cycle/_S)
= - \sum_{\emptyset \subsetneq S \subseteq I} \cycle|_S \cdot \apode_{T_1}(\cycle \!:\! T_1) \cdot \ldots \cdot \apode_{T_r}(\cycle \!:\! T_r),
\end{equation}
where~$T = T_1 \sqcup \dots \sqcup T_r$ is the partition of~$T:= I \setminus S$ into connected components of~$\cycle/_S = \cycle \!: \! T$.\break
Let~$\path_i$ denote the path~$\cycle \!:\! T_i$.
Then, an application of Proposition~\ref{p:apode_paths_tub} yields
\begin{equation}\label{eq:apode_cy_tu_aux1}
\apode_I(\cycle) = - \sum_{\emptyset \subsetneq S \subseteq I} \cycle|_S \cdot \prod_{i} \Big( \sum_{t^i} (-1)^{|t^i|} \path_i(t^i) \Big),
\end{equation}
where the internal sum is over all tubings~$t^i$ of~$\path_i$.

A choice of a nonempty subset~$S \subseteq I$ and of a tubing~$t^i$ for each~$\path_i$ is equivalent to a choice of a tubing~$t = \{I\} \sqcup t^1 \sqcup \dots \sqcup t^r$ of~$\cycle$. Indeed, given a tubing~$t$ of~$\cycle$, let~$T_1,\dots,T_r$ be the maximal proper tubes of~$t$. Then,~$S = I \setminus \big( T_1 \cup \dots \cup T_r)$ and~$t^i$ is the collection of tubes in~$t$ that are contained in~$T_i$. Moreover,
$$
S = \pi(t)_0,
\qquad
\pi(t)_+ = \pi(t^1) \sqcup \dots \sqcup \pi(t^r),
\qqand
\path_i|_{\pi^i_j} = (\cycle/_S)|_{\pi^i_j} \quad \text{for all } \pi^i_j \in \pi(t^i).
$$
Therefore,~$\cycle|_S \cdot \path_1(t^1) \dots \path_r(t^r) = \cycle(t)$ where~$t = \{I\} \sqcup t^1 \sqcup \dots \sqcup t^r$. Since~$|t| = 1 + |t_1| + \dots + |t^r|$,~\eqref{eq:apode_cy_tu_aux1} is equivalent to
$$
\apode_I(\cycle) = - \sum_t (-1)^{|t^1|+\dots+|t^r|} \cycle|_S \cdot \path_1(t^1) \dots \path_r(t^r) = \sum_t (-1)^{|t|} \cycle(t),
$$
as we wanted to show.

\end{proof}

In order to simplify notation, let us consider the standard cycle~$\cycle_n = (12\dots n)$. Let~$\pi = (\pi_0,\pi_+)$ be a pointed noncrossing partition of~$\cycle_n$ such that~$1 \in \pi_0$.
The \textbf{adjacent closure} of~$\pi = (\pi_0,\pi_+)$ is the pointed noncrossing partition~$\opp{\pi} = (\pi_0,\opp{\pi}_+)$,
where~$\opp{\pi}_+$ is obtained by successively merging non-zero blocks~$\pi_i,\pi_j$ if~$\max \pi_i + 1 = \min \pi_j$.

Again, this generalizes to any cycle~$\cycle$ and pointed noncrossing partition, as the following example illustrates.

\begin{example}
Consider the cycle~$\cycle$ and pointed noncrossing partition~$\pi$ in Example \ref{ex:PNC1}.
The adjacent closure~$\opp{\pi}$ of~$\pi$ is shown in the following picture.
$$
\includegraphics[scale=.8,page=2]{PNC-ex.pdf}
$$
\end{example}

\begin{theorem}\label{t:apode_cycles_Cat}
The following is a cancellation-free and grouping-free formula for the antipode of a cycle in~$\hc$.
If~$\cycle$ is a cycle on vertex set~$I$,
\begin{equation}\label{eq:apode_cycles_Cat}
\apode_I(\cycle) = \sum_{\pi \in PNC(\cycle)} (-1)^{|\pi|} C_{(\opp{\pi}_+:\pi_+)} \cycle(\pi),
\end{equation}
where~$PNC(\cycle)$ is the set of all pointed noncrossing partitions of~$\path$ and~$|\pi| = 1 + |\pi_+|$.
\end{theorem}

\begin{proof}

An application of Theorem~\ref{t:apode_paths_cat} to~\eqref{eq:apode_fst_step} yields
\begin{equation}\label{eq:apode_cy_ncp_aux1}
\apode_I(\cycle)
= - \sum_{\emptyset \subsetneq S \subseteq I} \cycle|_S \cdot \prod_{i} \Big( \sum_{\pi^i \in NC(\path_i)} (-1)^{|\pi^i|}C_{(\opp{\pi^i}:\pi^i)} \path_i(\pi^i) \Big),
\end{equation}
where~$\path_i = \cycle \! : \! T_i$ and~$T_1 \sqcup \dots \sqcup T_r$ is the partition into maximal intervals of~$\cycle\!:\!T$.

A choice of a nonempty subset~$S \subseteq I$ and of a noncrossing partition~$\pi^i$ for each~$\path_i$ is equivalent to a choice of a pointed noncrossing partition~$\pi = (S, \bigcup_i \pi^i)$ of~$\cycle$. Indeed, given a pointed noncrossing partition~$\pi$ of~$\cycle$, let~$S = \pi_0$ and~$T_1,\dots,T_r$ partition of~$I \setminus S$ into maximal intervals in~$\cycle$. Then,~$\pi^i = (\pi_+)|_{T_i}$ is the partition obtained by restricting~$\pi_+$ to sets contained in~$T_i$. Moreover, since blocks in different~$\pi^i$s are not adjacent ($S$ separates them), the adjacent closure of~$\pi$ is~$\opp{\pi} = (S, \bigcup_i \opp{\pi^i})$, and~$C_{(\opp{\pi}_+:\pi_+)} = \prod_i C_{(\opp{\pi^i}:\pi^i)}$.

Using these identifications, we can rewrite~\eqref{eq:apode_cy_ncp_aux1} as
$$
\apode_I(\cycle) = \sum_{\pi \in PNC(\cycle)} (-1)^{|\pi|} C_{(\opp{\pi}_+:\pi_+)} \cycle|_{\pi_0} \cdot \path_1(\pi^1) \cdots \path_r(\pi^r).
$$
Finally,

the result follows by noting that~$\cycle(\pi) = \cycle|_{\pi_0} \sqcup \path(\pi^1) \sqcup \dots \sqcup \path(\pi^r)$.
\end{proof}

In the following sections we will see some interesting applications of these formulas, and their relation with faces of cyclohedra.

\section{The group of characters of~$\hc$}\label{s:chars_C}

In this section, we compute the group of characters~$\bbX(\hc)$ of the Hopf monoid of paths and cycles. The following results will be fundamental in achieving this goal.

\begin{lemma}\label{l:count_intervals}
Let~$\cycle$ be a cycle on~$n$ vertices. Given integers~$k \geq 2$ and~$j_1,j_2,\dots,j_{n-k+1} \geq 0$ such that
$$
j_1 + j_2 + \dots + j_{n-k+1} = k
\qqand
j_1 + 2 j_2 + \dots + (n-k+1) j_{n-k+1} = n,
$$
there are
$$
\dfrac{n \, (k-1)!}{j_1! \, j_2! \dots j_{n-k+1}!}
$$
different partitions of~$\cycle$ into~$k$ vertex disjoint intervals with precisely~$j_i$ of them having length~$i$.
\end{lemma}

\begin{proof}
Without loss of generality, assume~$\cycle = (12\dots n)$.
Fix a \emph{starting position}~$i_0 \in [n]$,
and a total order~$\ell = (l_1,l_2,\dots,l_k)$ of the multiset of prescribed interval lengths~$\{1^{j_1},2^{j_2},\dots,(n-k+1)^{j_{n-k+1}}\}$.
The choice of~$i_0$ and~$\ell$ determines the following collection of vertex disjoint intervals of~$\cycle$:
$$
[i_0,i_0+l_1-1],[i_0+l_1,i_0+l_1+l_2-1],\dots,[i_0-l_k,i_0-1],
$$
where the operations are taken modulo~$n$.
There are
\begin{equation}\label{eq:aux_lem1}
n \cdot {k \choose j_1, j_2, \dots, j_{n-k+1}} = \dfrac{n \, k!}{j_1! \, j_2!\dots j_{n-k+1}!}
\end{equation}
possible choices for~$i_0$ and linear order~$\ell$.
However, any choice of intervals with the prescribed length appears exactly~$k$ times in this construction,
since any cyclic shift of~$\ell$ with the corresponding shift in the starting position~$i_0$ yields the same collection of intervals.
For example, choosing starting position~$i_0+l_1$ and order~$\ell' = (l_2,\dots,l_k,l_1)$ gives the same intervals as before, but listed in different order (a cyclic shift).
Dividing~\eqref{eq:aux_lem1} by~$k$ yields the desired result.
\end{proof}

We want to relate this result with nontrivial decompositions~$I = S \sqcup T$ of the vertex set of a cycle~$\cycle$. Observe that the maximal intervals in~$\cycle \!:\! T$ are never adjacent, otherwise they would form a larger interval. Thus, the total number of maximal intervals in~$\cycle \!:\! T$ is no more than~$|S|$ and no interval can contain more than~$|I| - |S|$ vertices.

\begin{proposition}\label{p:count_decomp}
Let~$\cycle \in \hc[I]$ be a cycle and~$n=|I|$.
Given integers~$k \geq 1$ and~$j_2,\dots,j_{n-k+1} \geq 0$ such that
$$
j_2 + \dots + j_{n-k+1} \leq k
\qqand
k + 1 j_2 + 2 j_3 + \dots + (n-k) j_{n-k+1} = n,
$$
there are
\begin{equation}\label{eq:aux_prop_decomp}
\dfrac{n \, (k-1)!}{j_1! \, j_2! \dots j_{n-k+1}!},
\end{equation}
decompositions~$I = S \sqcup T$ such that~$|S| = k$ and~$\cycle \! : \! T$ splits in maximal intervals with multiset of lengths~$\{ 1^{j_2},2^{j_3},\dots \}$, where~$j_1 = k - (j_2 + j_3 + \dots + j_{n-k+1})$.
\end{proposition}

\begin{proof}
Let us first deal with the case~$k=1$. In this situation, the inequalities above imply that $j_{n-k+1} = 1$ and~$j_1 = \dots = j_{n-k} = 0$. The expression \eqref{eq:aux_prop_decomp} takes value~$n$, which agrees with all the possible choices for~$S$. From now on we assume~$k \geq 2$.

We will construct a bijection between the decompositions~$I = S \sqcup T$ with the conditions above and configurations of~$k$ intervals in~$\cycle$ with multiset of lengths~$\{1^{j_1},2^{j_2},\dots,(n-k+1)^{j_{n-k+1}}\}$. The result will then follow from Lemma~\ref{l:count_intervals}. The following is an example of the bijection we will construct.

\begin{figure}[ht]
$$\begin{gathered}\includegraphics[scale=.7,page=1]{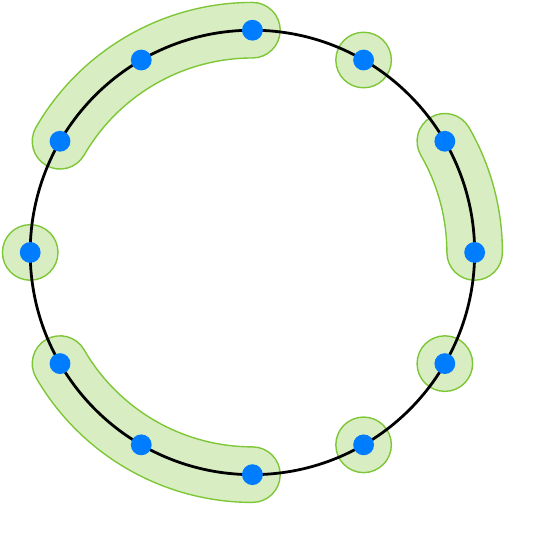}\end{gathered}
\longrightarrow
\begin{gathered}\includegraphics[scale=.7,page=2]{ex_biject2}\end{gathered}
\longrightarrow
\begin{gathered}\includegraphics[scale=.7,page=3]{ex_biject2}\end{gathered}$$
\caption{An example of the bijection described below. A partition into $7$ intervals corresponds to a decomposition $I = S \sqcup T$ where $|S| = 7$.}
\end{figure}

Fix an orientation of~$\cycle$, in the example above it is the clockwise orientation. Given a collection of intervals~$\{I_1,\dots,I_k\}$ of~$\cycle$ with multiset of lengths~$\{1^{j_1},2^{j_2},\dots,(n-k+1)^{j_{n-k+1}}\}$, let~$S$ be the collection of ``right endpoints'' (the last element according to the fixed orientation) of these intervals and~$T = I \setminus S$. Since we have chosen exactly one point for each interval, we have~$|S| = k$. An interval~$I_r$ with~$j \geq 2$ elements becomes a maximal interval of~$\cycle\!:\!T$ with~$j-1$ elements, since we have removed an endpoint and thus not disconnected the remaining vertices in~$I_r$.

Conversely, given such a decomposition~$I = S \sqcup T$, attach to each maximal interval of~$\cycle\!:\!T$ the \emph{next} element not in the interval. It will necessarily be an element of~$S$. The remaining~$k - (j_2 + \dots + j_{n-k+1}) = j_1$ elements of~$S$ become intervals of length~$1$.
\end{proof}

We now turn our attention to describing the group of characters of~$\bbX(\hc)$. Since~$\rmF$ is a Hopf monoid quotient of~$\hc$, its group of characters~$\bbX(\rmF)$ is a subgroup of~$\bbX(\hc)$.

\begin{theorem}[{\citep[Theorem 10.5]{aa17}}]\label{t:chars_F}
The group of characters~$\rmF$ is isomorphic to the group of series of the form
$$
g(x) = x + a_1 x^2 + a_2 x^3 + \dots
$$
under composition.
\end{theorem}

Let~$G$ be the group of pairs of power series~$\big(g(x),h(x)\big)$ of the form
$$
g(x) = x + a_1 x^2 + a_2 x^3 + \dots
\qqand
h(x) = c_1 x + c_2 x^2 + c_3 x^3 + \dots
$$
under the following operation
$$
(g_1,h_1) \cdot (g_2,h_2) = (g_1 \circ g_2 , h_1 \circ g_2 + h_2).
$$
One easily verifies that this operation is associative, that~$(x,0)$ is the unit of~$G$, and that~$(g,h)^{-1} = (g^{\langle -1 \rangle},-h \circ g^{\langle -1 \rangle})$, where~$g^{\langle -1 \rangle}$ denotes the compositional inverse of~$g$. Observe that~$\bbX(\rmF)$ naturally embeds~$G$ by means of the map~$g(x) \mapsto \big( g(x), 0\big)$.

\begin{theorem}\label{t:chars_hc}
The group of characters of~$\hc$ is isomorphic to~$G$.
\end{theorem}

Before diving into the proof of this theorem, let us recall the Faà di Bruno Formula~\citep[Chapter 2, Section 8, Formula (45a)]{riordan} for the composition of power series. Let
$$
f(x) = f_1 x + f_2 x^2 + f_3 x^3 + \dots
\qqand
g(x) = g_1 x + g_2 x^2 + g_3 x^3 + \dots
$$
be two power series and~$h = f \circ g$. The coefficients of~$h$ are determined by
$$
h_n = \sum_{1 \leq k \leq n} f_k \widehat{B}_{n,k}(g_1,g_2,\dots,g_{n-k+1}),
$$
where~$\widehat{B}_{n,k}$ is the ordinary Bell polynomial
$$
\widehat{B}_{n,k}(x_1,x_2,\dots,x_{n-k+1})
= \sum \dfrac{k!}{j_1! \, j_2! \dots j_{n-k+1}!} x_1^{j_1} x_2^{j_2} \dots x_{n-k+1}^{j_{n-k+1}}.
$$
The sum is over all nonnegative integers~$j_1,\dots,j_{n-k+1}$ such that
$$
j_1 + j_2 + \dots + j_{n-k+1} = k
\qqand
1\, j_1 + 2\, j_2 + \dots + (n-k+1)\, j_{n-k+1} = n.
$$

\begin{proof}[{Proof of Theorem~\ref{t:chars_hc}}]
All paths on the same number of vertices are isomorphic, the same is true for cycles. Therefore, a character~$\zeta \in \bbX(\hc)$ is determined by the values is takes on the graphs~$\path_n = 12\dots n$ and~$\cycle_n = (12\dots n)$ for~$n \geq 1$. We define a function~$\bbX(\hc) \rightarrow G$ by
$$
\zeta \mapsto \Big( x + \zeta_{[1]}( \path_1 )x^2 + \zeta_{[2]}( \path_2 )x^3 + \dots
\,\,\, ,\,\,\,
\zeta_{[1]}( \cycle_1 )x + \zeta_{[2]}( \cycle_2 ) \dfrac{x^2}{2} + \zeta_{[3]}( \cycle_3 ) \dfrac{x^3}{3} + \dots \Big).
$$
Since power series are uniquely determined by its coefficients, this map is clearly a bijection. We proceed to prove that it is a group homomorphism.

Let~$\zeta,\xi$ be two characters of~$\hc$, with corresponding pairs of power series
$$
(g_1(x),h_1(x)) = \Big( x + \sum_{n\geq 1} a_n x^{n+1},\sum_{n\geq 1} c_n \dfrac{x^n}{n} \Big)
\qqand
(g_2(x),h_2(x)) = \Big( x + \sum_{n\geq 1} \alpha_n x^{n+1},\sum_{n\geq 1} \gamma_n \dfrac{x^n}{n} \Big),
$$
where
$$
a_n = \zeta(\path_n)
\qquad
c_n = \zeta(\cycle_n)
\qquad\qquad
\alpha_n = \xi(\path_n)
\qquad
\gamma_n = \xi(\cycle_n).
$$
We want to show that the pair of power series associated to~$\zeta*\xi$ is~$(g_1 \circ g_2 , h_1 \circ g_2 + h_2)$. The result in the first entry of the pair is precisely Theorem~\ref{t:chars_F}. We proceed to explicitly compute the second power series associated to~$\zeta*\xi$.
\begin{equation}\label{eq:convol_gamma}
\begin{array}{RL}
\dfrac{(\zeta*\xi)(\cycle_n)}{n}
& = \dfrac{1}{n} \sum_{[n] = S \sqcup T} \zeta(\cycle_n|_S) \xi( \cycle_n/_S ) \\
& = \dfrac{1}{n}\sum_{ \emptyset \neq S \subseteq [n]} c_{|S|} \alpha_{|T_1|} \dots \alpha_{|T_r|} + \dfrac{\gamma_n}{n},
\end{array}
\end{equation}
where~$T = T_1 \sqcup \dots \sqcup T_r$ is the natural decomposition of~$[n] \setminus S$ into maximal intervals in~$\cycle_n$.
Grouping similar terms and using Proposition~\ref{p:count_decomp}, the coefficient of~$c_k \alpha_1^{j_2} \alpha_2^{j_3} \dots \alpha_{n-k}^{j_{n-k+1}}$ in the sum above is
$$
\dfrac{n \, (k-1)!}{j_1! \, j_2!\dots j_{n-k+1}!}.
$$
Thus, we can rewrite~\eqref{eq:convol_gamma} as
$$
\begin{array}{RL}
& \dfrac{1}{n} \bigg( \sum_{\substack{j_1 + j_2 + \dots + j_{n-k+1} = k \\ j_1 + 2 j_2 + \dots (n-k+1) j_{n-k+1} = n}}
\dfrac{n \, (k-1)!}{j_1! \, j_2!\dots j_{n-k+1}!}
c_k \alpha_1^{j_2} \alpha_2^{j_3} \dots \alpha_{n-k}^{j_{n-k+1}} \bigg) + \dfrac{\gamma_n}{n}\\
= & \sum_{1 \leq k \leq n} \dfrac{c_k}{k} \hat{B}_{n,k}(1,\alpha_1,\dots,\alpha_{n-k}) + \dfrac{\gamma_n}{n},
\end{array}
$$
which, by the Faà di Bruno Formula, is precisely the coefficient of~$x^n$ in~$h_1 \circ g_2 + h_2$.
\end{proof}

\section{Associahedra, cyclohedra and inversion in pairs of power series}\label{s:ass_cyc_ps}

Let $\opp{\rmA}$ denote the Hopf submonoid of $\opp{\GP}$ generated by the classes of associahedra, and $\opp{\hc}$ the Hopf submonoid generated by the classes of associahedra and cyclohedra. The Hopf monoid morphism $\hc \rightarrow \opp{\GP}$ sending each graph $g \in \hc[I]$ to the class of its graph associahedron $\ass_g \in \opp{\GP}$ is not injective, since $\ass_{\path_1} = \ass_{\cycle_1}$ and $\ass_{\path_2} = \ass_{\cycle_2}$. The following diagram summarizes the relations between the Hopf monoids we have defined so far.
$$
\begin{tikzcd}[row sep=tiny]
\hc \arrow[r, two heads] & \opp{\hc} \arrow[rd, hook] & \\
 & & \opp{\GP} \\
\rmF \arrow[r, "\cong"'] \arrow[uu, hook] & \opp{\rmA} \arrow[uu, hook] \arrow[ru, hook] &
\end{tikzcd}$$
The isomorphism $F \xrightarrow{\,\cong\,} \rm{A}$ was first described by Aguiar and Ardila~\citep[Proposition 25.7]{aa17}. The Hopf monoid $\opp{\hc}$ is the quotient of $\hc$ obtained by identifying (the graph associahedra of) $\path_n$ and $\cycle_n$ for $n=1,2$.

\begin{corollary}
The group of characters of $\opp{\hc}$ is isomorphic to the subgroup of $G$ consisting of pairs of power series
$$
( x + a_1x^2 + a_2 x^3 + a_3 x^4 + \dots
\,\,\, ,\,\,\,
c_1x + c_2\dfrac{x^2}{2} + c_3\dfrac{x^3}{3} + \dots) \in G
$$
such that $a_1 = c_1$ and $a_2 = c_2$.
\end{corollary}

The standard associahedron~$\ass_n$ is the graph associahedron of the path~$\path_n = 12\dots n$, this is Loday's realization of the associahedron.
Analogously, the standard cyclohedron~$\cyc_n$ is the graph associahedron of the cycle~$\cycle_n = (12\dots n)$.
The description of the faces of~$\ass_n$ and~$\cyc_n$ in terms of tubings (see Section~\ref{ss:g_ass}) and the definitions in Section~\ref{s:apode_C} show the following.

\begin{itemize}
\item Let~$t$ be a tubing of~$\path_n$ and~$\pi(t) = \{ \pi_1 , \dots , \pi_k \}$ be the associated noncrossing partition of~$[n]$.
		The graph~$\path_n(t)$ is the disjoint union of the~$k$ paths~$\path^i := \path_n|_{\pi_i}$, one for each block~$\pi_i \in \pi(t)$.
		Thus, the face~$F_t \leq \ass_n$ is normally equivalent to the following product of associahedra:
		$$
		\ass_{\path^1} \times \dots \times \ass_{\path^k},
		$$
		which in turn is isomorphic to the following product of \emph{standard} associahedra
		$$
		\ass_{|\pi_1|} \times \dots \times \ass_{|\pi_k|}.
		$$
		Therefore, every face of~$\ass_n$ naturally decomposes as the product of (polytopes isomorphic to) other standard associahedra of lower dimension.
		
\item Let~$t$ be a tubing of~$\cycle_n$ and~$\pi(t) = (\pi_0, \{\pi_1 , \dots , \pi_k \})$ be the associated pointed noncrossing partition of~$[n]$.

		Then,~$\cycle_n(t)$ is the disjoint union of a cycle~$\cycle^0 = \cycle_n|_{\pi_0}$ in~$\pi_0$ and a path~$\path^i = (\cycle_n/_{\pi_0})|_{\pi_i}$ on each~$\pi_i \in \pi(t)_+$.
		The face~$F_t \leq \cyc_n$ is normally equivalent to the following product of graph associahedra:
		$$
		\ass_{\cycle^0} \times \ass_{\path^1} \times \dots \times \ass_{\path^k},
		$$
		which in turn is isomorphic to the following product of a \emph{standard} cyclohedron and \emph{standard} associahedra
		$$
		\cyc_{|\pi_0|} \times \ass_{|\pi_1|} \times \dots \times \ass_{|\pi_k|}.
		$$
\end{itemize}

Aguiar and Ardila use the previous description of the faces of an associahedron along with their antipode formula for generalized permutahedra (Theorem~\ref{t:apode_GP}) and the description of the group of characters of $\rmF \cong \opp{A}$ (Theorem~\ref{t:chars_F}) to describe inversion in the group of power series under composition, in terms of the faces of associahedra.

\begin{theorem}[{\citep[Theorem 11.3]{aa17}}]\label{t:polyt_Linv}
The compositional inverse of
$$
g(x) = x + a_1 x^2 + a_2 x^3 + \dots
\qquad \text{is} \qquad
g^{\langle -1 \rangle}(x) = x + b_1 x^2 + b_2 x^3 + \dots,
$$
where
\begin{equation}\label{eq:polyt_Linv}
b_n = \sum_{F \leq \ass_n} (-1)^{n - \dim F} a_F
\end{equation}
and we write~$a_F = a_{f_1} \dots a_{f_k}$ for each face~$F \cong \ass_{f_1} \times \dots \times \ass_{f_k}$ of the associahedron~$\ass_n$.
\end{theorem}

\begin{example}\label{ex:Linv_Echar}
A direct computation shows that the compositional inverse of
$$
\dfrac{x}{1-x} = x + x^2 + x^3 + x^4 + \dots
\qquad \text{is} \qquad
\dfrac{x}{1+x} = x - x^2 + x^3 - x^4 + \dots
$$
The identity~\eqref{eq:polyt_Linv} in this case reads
$$
(-1)^n = (-1)^n \sum_{F \leq \ass_n} (-1)^{\dim F}.
$$
The sum on the right hand side is computing the Euler characteristic of~$\ass_n$, which we know is~$1$.
\end{example}

The compositional inverse of a power series can be described with a formula having fewer terms. This formula can be obtained by grouping the terms in~\eqref{eq:polyt_Linv} corresponding to isomorphic faces of~$\ass_n$, or equivalently, by using the description of the antipode of paths in Theorem~\ref{t:apode_paths_cat}. In doing so, we obtain the following formula for the coefficients of~$g^{\langle -1 \rangle}$.

\begin{theorem}
The coefficients of the compositional inverse $g^{\langle -1 \rangle}$ of $g(x)$ are determined by
\begin{equation}\label{eq:Cat_Linv}
b_n = \sum_{\pi \in NC(\path_n)} (-1)^{|\pi|}C_{(\opp{\pi}:\pi)}a_{|\pi_1|} a_{|\pi_2|} \dots a_{|\pi_r|}
\end{equation}
\end{theorem}

\begin{example}
We can explicitly verify that the compositional inverse of
$$
e^x - 1 = x + \dfrac{x^2}{2!} + \dfrac{x^3}{3!} + \dfrac{x^4}{4!} + \dots
\qquad \text{is} \qquad
\ln(1+x) = x - \dfrac{x^2}{2} + \dfrac{x^3}{3} - \dfrac{x^4}{4!} + \dots
$$
Then, the expression~\eqref{eq:Cat_Linv} yields the following combinatorial formula involving Catalan numbers:
$$
\sum_{\pi \in NC(\path_n)} (-1)^{|\pi|} \dfrac{C_{(\opp{\pi}:\pi)}}{(|\pi_1|+1)! \, (|\pi_2|+1)! \dots (|\pi_r|+1)!} = \dfrac{(-1)^n}{n+1}.
$$
For instance,~$\{1,2,3\}$,~$\{12,3\}$,~$\{13,2\}$,~$\{1,23\}$,~$\{123\}$ are all the (noncrossing) partitions of~$[3]$. Listing the terms in that order, the previous formula reads
$$
- \dfrac{5}{(2!)^3} + \dfrac{2}{3!\, 2!} + \dfrac{1}{3!\, 2!} + \dfrac{2}{3!\, 2!} - \dfrac{1}{4!} = - \dfrac{1}{4}.
$$
\end{example}

We proceed to extend formulas~\eqref{eq:polyt_Linv} and~\eqref{eq:Cat_Linv} to pairs of power series in~$\bbX(\hc) \cong G$. For the analogous of formula~\eqref{eq:polyt_Linv}, we would expect to obtain a sum over the faces of the cyclohedron, but there is one small caveat. Let us consider, for example, a two-dimensional face~$F$ of~$\cyc_5$ that is a (combinatorial) square.
$F$ is isomorphic to both~$\cyc_2 \times \ass_2 \times \ass_1$ and~$\cyc_1 \times \ass_2 \times \ass_2$,
so how do we determine if the corresponding term~$c_F$ in the analogous formula is~$c_2a_2a_1$ or~$c_1a_2^2$?
The following result answers this question.

\begin{theorem}\label{t:cyc_inv}
The inverse of the pair
$$
\Big(
x + a_1 x^2 + a_2 x^3 + \dots
\, , \,
c_1 x + c_2 \dfrac{x^2}{2} + c_3 \dfrac{x^3}{3} + \dots
\Big)
$$
in the group~$G$ is
$$
\Big(
x + b_1 x^2 + b_2 x^3 + \dots
\, , \,
d_1 x + d_2 \dfrac{x^2}{2} + d_3 \dfrac{x^3}{3} + \dots
\Big),
$$
where~$b_n$ is determined by the faces of the associahedron~$\ass_n$ as in~\eqref{eq:polyt_Linv},
\begin{equation}\label{eq:cyc_inv}
d_n = \sum_{F \leq \cyc_n} (-1)^{n - \dim F} c_F,
\end{equation}
and we write~$c_F = c_{f_0} a_{f_1} \dots a_{f_k}$ for each face~$F \cong \cyc_{f_0} \times \ass_{f_1} \times \dots \times \ass_{f_k}$.
The Cartesian factor~$\cyc_{f_0}$ of~$F$

is determined by having vertices with one coordinate attaining the value~${n \choose 2} + 1$.
\end{theorem}

\begin{proof}
The statement about the coefficients~$b_n$ follows directly from Theorem~\ref{t:polyt_Linv}.Let~$\zeta \in \bbX(\hc)$ be the character corresponding to the pair~$(x + \sum a_n x^{n+1} , \sum c_n \tfrac{x^n}{n})$. Then, using Theorem~\ref{t:inverse_chars} we have
$$
d_n = (\zeta_{[n]} \circ \apode_{[n]})(\cycle_n).
$$
Using Theorem~\ref{t:apode_cycles_tub} to compute the antipode of~$\cycle_n$, and the linearity and multiplicativity of~$\zeta_I$, we get
$$
\begin{array}{RL}
d_n
&= \sum_t (-1)^{|t|} \zeta_{[n]}(\cycle_n(t)) \\
&= \sum_t (-1)^{|t|} \zeta_{\pi_0}(\cycle^0)\zeta_{\pi_1}(\path^1)\dots\zeta_{\pi_k}(\path^k)\\
&= \sum_t (-1)^{|t|} c_{|\pi_0|}a_{|\pi_1|} \dots a_{|\pi_k|}.
\end{array}
$$
The sums are over all tubings of~$\cycle_n$, and~$(\pi_0,\{\pi_1,\dots,\pi_k\}) = \pi(t)$.
The bijection between tubings of~$\cycle_n$ and faces of~$\cyc_n = \ass_{\cycle_n}$ allows us to rewrite
$$
(-1)^{|t|} c_{|\pi_0|}a_{|\pi_1|} \dots a_{|\pi_k|} = (-1)^{n - \dim F_t} c_{f_0}a_{f_1} \dots a_{f_k},
$$
where~$f_i = |\pi_i|$ and~$F_t \equiv \ass_{\cycle^0} \times \ass_{\path^1} \times \dots \times \ass_{\path^k} \cong \cyc_{f_0} \times \ass_{f_1} \times \dots \times \ass_{f_k}$.
Recall that~$F_t \cong \cyc_{f_0} \times \ass_{f_1} \times \dots \times \ass_{f_k}$ is the face of~$\cyc_n$ maximized by the linear functional~$- \sum_i n_i x_i$, where~$n_i$ is the number of tubes of~$t$ containing~$i \in [n]$.
Property~\ref{p:greedy} implies that~$F_t$ is contained in the face of~$\cyc_n$ maximized by~$\sum_{i \in \pi_0} x_i$.
This face is
$$
\Delta_{\pi_0} + \sum_{[i,j] \cap \pi_0 \neq \emptyset} \Delta_{[i,j] \cap \pi_0} + \sum_{[i,j] \cap \pi_0 = \emptyset} \Delta_{[i,j]}
$$
where the sums are over proper intervals~$[i,j] = \{i,i+1,\dots,j\}$ with~$j \neq i-1$, the operations are modulo~$n$.
For~$l \in \pi_0$, the maximum value of~$x_l$ in this face counts the number of tubes of~$\cyc_n$ containing~$l$:
\begin{equation}\label{eq:aux_max_face_cyc}
\#\{\text{tubes of } \cyc_n\} - \#\{\text{tubes not containing } l\} = \big( n(n-1) + 1 \big) - {n \choose 2} = {n \choose 2} + 1.
\end{equation}
On the other hand, for~$m \notin \pi_0$, the maximum value of~$x_m$ in this face counts the number of intervals~$[i,j]$ containing~$m$ that do not intersect~$\pi_0$.
Since~$\pi_0 \neq \emptyset$, this is strictly less than~\eqref{eq:aux_max_face_cyc}.
Therefore~$\pi_0$, and thus the Cartesian factor~$\cyc_{f_0}$, is determined by consisting of those coordinates that reach the value~${n \choose 2} + 1$ at some vertex of~$F$.
\end{proof}

\begin{example}
Extending Example~\ref{ex:Linv_Echar}, consider the pair
$$
\Big(\dfrac{x}{1-x} \, , \, -\ln(1-x)\Big) =
\Big(
x + x^2 + x^3 + \dots
\, , \,
x + \dfrac{x^2}{2} + \dfrac{x^3}{3} + \dots
\Big),
$$
whose inverse in~$G$ is
$$
\Big(\dfrac{x}{1+x} \, , \, -\ln(1+x)\Big)
=
\Big(
x - x^2 + x^3 - \dots
\, , \,
-x + \dfrac{x^2}{2} - \dfrac{x^3}{3} + \dots
\Big)
.
$$
In this case, identity~\eqref{eq:cyc_inv} reflects that the Euler characteristic of the cyclohedron is~$1$.
\end{example}

By grouping terms in~\eqref{eq:cyc_inv}, or equivalently by applying the character~$\zeta$ in the proof of Theorem~\ref{t:cyc_inv} to both sides of~\eqref{eq:apode_cycles_Cat}, we obtain a formula of the coefficients~$d_n$ with fewer terms.

\begin{theorem}
The coefficients of the second component of $(g,h)^{-1}$ for $(g,h) \in G$ are determined by
\begin{equation}\label{eq:Cat_cyc_inv}
d_n = \sum_{\pi \in PNC(\cycle_n)} (-1)^{|\pi|}C_{(\opp{\pi}_+:\pi_+)} c_{|\pi_0|} a_{|\pi_1|} a_{|\pi_2|} \dots a_{|\pi_r|}
\end{equation}
\end{theorem}

We conclude by using formula~\eqref{eq:Cat_cyc_inv} to deduce some enumerative identities involving pointed noncrossing partitions and Catalan numbers.

\begin{example}
Consider the pair
$$
\Big(\dfrac{x}{1-x} \, , \, \dfrac{x}{1-x}\Big)=
\Big(
x + x^2 + x^3 + \dots
\, , \,
x + 2\dfrac{x^2}{2} + 3\dfrac{x^3}{3} + \dots
\Big),
$$
and its inverse
$$
\Big(\dfrac{x}{1+x} \, , \, - x\Big)
=
\Big(
x - x^2 + x^3 - \dots
\, , \,
-x
\Big)
.
$$
In this case, identity~\eqref{eq:Cat_cyc_inv} translates into the following identity involving pointed noncrossing partitions and Catalan numbers:
$$
\sum_{\pi \in PNC(\cycle_n)} (-1)^{|\pi_+|} |\pi_0| C_{(\opp{\pi}_+:\pi_+)}
= \begin{cases}
1 & \text{if } n = 1,\\
0 & \text{otherwise}.
\end{cases}
$$
\end{example}

\begin{example}
A similar analysis with the pair~$\big(e^x -1 \, , \, e^x -1 \big)$ and its inverse~$\big( \ln(1+x) \, , \, -x \big)$ yields the following combinatorial identity:
$$
\sum_{\pi \in PNC(\cycle_n)} (-1)^{|\pi_+|} \dfrac{C_{(\opp{\pi}_+:\pi_+)}}{(|\pi_0|-1)! \, (|\pi_1|+1)! \dots (|\pi_r|+1)!}
= \begin{cases}
1 & \text{if } n = 1,\\
0 & \text{otherwise}.
\end{cases}
$$
\end{example}

\bibliographystyle{amsalpha}
\bibliography{../../bib}

\providecommand{\bysame}{\leavevmode\hbox to3em{\hrulefill}\thinspace}
\providecommand{\MR}{\relax\ifhmode\unskip\space\fi MR }
\providecommand{\MRhref}[2]{%
  \href{http://www.ams.org/mathscinet-getitem?mr=#1}{#2}
}
\providecommand{\href}[2]{#2}
\begin{thebibliography}{PRW08}

\bibitem[AA17]{aa17}
Marcelo Aguiar and Federico Ardila, \emph{Hopf monoids and generalized
  permutahedra}, arXiv preprint arXiv:1709.07504 (2017).

\bibitem[AM10]{am10}
Marcelo Aguiar and Swapneel Mahajan, \emph{Monoidal functors, species and
  {H}opf algebras}, CRM Monograph Series, vol.~29, American Mathematical
  Society, Providence, RI, 2010, With forewords by Kenneth Brown and Stephen
  Chase and Andr{\'e} Joyal.

\bibitem[AM13]{am13}
\bysame, \emph{Hopf monoids in the category of species}, Hopf algebras and
  tensor categories, Contemp. Math., vol. 585, American Mathematical Society,
  Providence, RI, 2013, pp.~17--124.

\bibitem[CD06]{cd06grass}
Michael~P. Carr and Satyan~L. Devadoss, \emph{Coxeter complexes and
  graph-associahedra}, Topology Appl. \textbf{153} (2006), no.~12, 2155--2168.

\bibitem[Edm70]{edmonds}
Jack Edmonds, \emph{Submodular functions, matroids, and certain polyhedra},
  Combinatorial {S}tructures and their {A}pplications ({P}roc. {C}algary
  {I}nternat. {C}onf., {C}algary, {A}lta., 1969), Gordon and Breach, New York,
  1970, pp.~69--87.

\bibitem[Joy81]{joyal}
Andr{\'e} Joyal, \emph{Une th{\'e}orie combinatoire des s{\'e}ries formelles},
  Adv. in Math. \textbf{42} (1981), no.~1, 1--82.

\bibitem[Pos09]{postnikov09}
Alexander Postnikov, \emph{Permutohedra, associahedra, and beyond}, Int. Math.
  Res. Not. IMRN (2009), no.~6, 1026--1106.

\bibitem[PRW08]{prw08faces}
Alex Postnikov, Victor Reiner, and Lauren Williams, \emph{Faces of generalized
  permutohedra}, Doc. Math. \textbf{13} (2008), 207--273.

\bibitem[Rio58]{riordan}
John Riordan, \emph{An introduction to combinatorial analysis}, Wiley
  Publications in Mathematical Statistics, John Wiley \& Sons, Inc., New York;
  Chapman \& Hall, Ltd., London, 1958.

\end{thebibliography}

\end{document}